\documentclass{amsart}
\usepackage{amssymb}
\usepackage{mathrsfs}
\usepackage{graphicx}
\usepackage{comment}
\usepackage{rotating}

\newcommand{\ve}{\varepsilon}

\newcommand{\nni}{\noindent}
\newcommand{\be}{\begin{equation}}
\newcommand{\ee}{\end{equation}}
\newcommand{\ba}{\begin{align}}
\newcommand{\ea}{\end{align}}

\newcommand{\abs}[1]{\lvert#1\rvert}

\DeclareMathOperator{\Res}{Res}

\newtheorem{theorem}{Theorem}[section]

\newtheorem{lemma}{Lemma}[section]

\newtheorem{question}{Question}

{\begin{list}{}{%
\settowidth{\labelwidth}{\textsf{{\it #1.}}}%
\setlength{\labelsep}{4mm}%
\setlength{\leftmargin}{\labelwidth}%
\addtolength{\leftmargin}{\labelsep}%
}}%
{\end{list}}

\def\om{\omega}
\def\beq{\begin{equation}}\def\enq{\end{equation}}

{\begin{list}{}{%
\settowidth{\labelwidth}{\textsf{{\it #1.}}}%
\setlength{\labelsep}{2mm}%
\setlength{\leftmargin}{\labelwidth}%
\addtolength{\leftmargin}{\labelsep}%
\addtolength{\leftmargin}{4mm}%
\setlength{\itemsep}{6pt}%
\setlength{\listparindent}{0pt}%
\setlength{\topsep}{3pt}%
}}%

\title[Integer group determinants ]{The integer group determinants for GA(1,p)
and related semidirect products.}

\author[B. Paudel]{Bishnu Paudel}
\address{Mathematics and Statistics Department\\
University of Minnesota Duluth\\
Duluth, MN 55812, USA}
\email{bpaudel@d.umn.edu}

\author[C. Pinner]{Chris Pinner}
\address{ Department of Mathematics\\
         Kansas State University\\
         Manhattan, KS 66506, USA}
\email{pinner@math.ksu.edu}

\keywords{Integer group determinants, small groups, semidirect products.}
\subjclass[2010]{Primary: 11C20, 15B36; Secondary: 11C08, 43A40}

\date{\today}
\begin{document}

\begin{abstract}
We consider the integer group determinants for 
groups that are semidirect products of $\mathbb Z_p$ and $\mathbb Z_n$
with $p$ prime and $n\mid p-1$. We give a complete description of the integer group determinants for the general affine groups of degree one  GA(1,$p$) when $p=5$ (one of the unresolved non-abelian groups of order 20), $7,11$ and $23$, 
and for $\mathbb Z_7\rtimes \mathbb Z_3$ (the non-abelian group of order 21),
 $\mathbb Z_{11}\rtimes \mathbb Z_5$ and  $\mathbb Z_{13}\rtimes \mathbb Z_6,$
showing that the obvious  divisibility and congruence conditions arising from the form of the group determinant when $n=p-1$ or $\frac{1}{2}(p-1)$, can be sufficient 
as well as necessary for these types of groups (although in the latter case we must work with norms of integers in a quadratic field). For $p=13$ this also happens for the remaining groups of this type, $\mathbb Z_{13}\rtimes_5 \mathbb Z_4$ and $\mathbb Z_{13}\rtimes \mathbb Z_3$, (working in an appropriate cubic and quartic field).

\end{abstract}

\maketitle

\section{Introduction}
For a group $G$ of order $n$ the group determinant is a homogeneous polynomial of degree $n$  in $n$ variables $x_g$, one for each element $g\in G.$ 
At the 1977 AMS meeting in Hayward, California, Olga Taussky-Todd  \cite{TausskyTodd} asked what integer values 
a group determinant can take when the variables $x_g$ are all assigned integer values $a_g$.
Here we shall think of the  evaluated group determinant as being defined on an element in the group ring $\mathbb Z[G]$, with  $a_g$ as the coefficient of $g$:
\be \label{defD} D\left(\sum_{g\in G} a_g g\right): = \det\left( a_{gh^{-1}}\right),
\ee
where $g\in G$ indexes the rows, and $h\in G$  the columns of the matrix. Notice that for $\alpha,\beta $ in $\mathbb Z[G]$ we have
\be \label{mult} D(\alpha\beta)=D(\alpha)D(\beta). \ee
As observed by Frobenius \cite{Frob}, the group determinant can be factored in the form
\be \label{Frob} 
 D\left(\sum_{g\in G} a_g g\right)=\prod_{\rho\in \hat{G}} \det \left( \sum_{g\in G} a_g\rho(g)\right)^{\deg \rho}, \ee
where $\hat{G}$ denotes a full set of irreducible, non-isomorphic representations for $G$
(see \cite{Conrad} or \cite{book} for a survey of the history).

Taussky-Todd was most interested in the case of $\mathbb Z_n$, the cyclic group of order $n$, where the determinants are  the $n\times n$ circulant determinants (each row is a 
cyclical shift by one of the previous row).  Newman \cite{Newman1} and Laquer \cite{Laquer} showed that in the cyclic case, $G=\mathbb Z_n,$ the integer group determinants included all the integers
\be \label{cyclic} \{ m\; :\; \gcd(m,n)=1 \text{ or } n^2 \mid m\}, \ee
and also obtained sharp divisibility conditions; for a prime $q$
\be \label{divy} q^{k}\parallel n, \; q \mid D(\alpha) \Rightarrow q^{k+1}\mid D(\alpha), \ee
where this can be improved to $2^{k+2}\mid D(\alpha)$ when $q=2$ and $k\geq 2$ (see also \cite{Mahoney1} and \cite{Norbert2}). Moreover, in the case of $\mathbb Z_p$ and $\mathbb Z_{2p}$, with $p$ a prime, these necessary conditions \eqref{divy} are also sufficient. 

 Similarly 
sharp divisibility and congruence restrictions were obtained for $D_n$, the dihedral group of order $n$, in \cite{dihedral}  (see also \cite{Mahoney2}) with these conditions both necessary and sufficient for $D_{2p}$ and $D_{4p}$.  Unfortunately the cases where the divisibility and congruence conditions are sufficient seem quite rare and the situation quickly becomes complicated even in the cyclic  case; for example Newman \cite{Newman2} showed that this fails for $\mathbb Z_{p^2}$ once $p\geq 5$, and $\mathbb Z_{pq}$ is not straightforward even for $\mathbb Z_{15}$ and umanageable in general
 (see \cite{bishnu1} and \cite{Mike}).

The $\mathbb Z_p,\mathbb Z_{2p}, D_{2p},D_{4p}$ account for many of the small groups 
and a complete description of the  integer group determinants for all the remaining groups with $|G|\leq 14$ 
was given in \cite{smallgps}.  
A recent series of papers \cite{bishnu2,bishnu3,bishnu4,Yamaguchi0,Yamaguchi3,Yamaguchi1,Yamaguchi2,Yamaguchi4,Yamaguchi5,Yamaguchi678,Yamaguchi9}
and \cite{18} dealt with all the $|G|=16$ and 18, so that now all groups with $|G|<20$ are fully determined. 
For the nonabelian groups the list has been pushed a little further. The integer group determinants for $Q_{20}$, the dicyclic group of order 20, were mostly  (though not completely)
determined in \cite{bishnu0}.
This leaves one nonabelian group of order 20,  SmallGroup(20,3) or GA(1,5),
\be G=\langle X, Y \; | \; X^5=Y^4=1,\; YXY^{-1}=X^2\rangle \ee
to be considered here.
As we shall see this is one of those unusual cases where the divisibilty and congruence conditions arising immediately from the form of the group determinant turn out to be both necessary and sufficient.

The next smallest non-abelian group to consider would be SmallGroup(21,1)
$$ G=\langle X,Y \; : \; X^7=Y^3=1, \; YXY^{-1}=X^2 \rangle. $$

Clearly these two groups are both special cases of groups of the form
\be \label{general}  G=\langle X, Y \; | \; X^p=Y^n=1,\; YXY^{-1}=X^r\rangle, \quad \text{ord}_p(r)=n, \ee
for a prime $p$ and positive integer $n\mid p-1$.  Notice that these are semidirect products
$\mathbb Z_p \rtimes_{\theta} \mathbb Z_n$ where we can think of $\mathbb Z_p$ as the
finite field $\mathbb F_p$ and $\mathbb Z_n$ as the multiplicative subgroup $\langle r \rangle$ of $\mathbb F_p^*$ of order $n$, where the automorphisms $\theta(r^j)$ of $\mathbb F_p$ are just multiplication by $r^j$. Notice the special case $n=p-1$ corresponds
to the dimension one general affine group $GA(1,p)$. The case $n=2$ corresponds to the dihedral group $D_{2p}$. When $n<p$ are both primes, recall that this is the unique non-abelian group of order $np$ if $n\mid p-1$, and that there are none if $n\nmid p-1$ (e.g.\ Dummit \& Foote \cite[p183]{Dummit}). The abelian groups of order 24 are tackled in \cite{24}.

\section{The form of the group determinant}\label{formulaH}
Suppose that $G$ has the form:
\be  G=\langle X, Y \; | \; X^p=Y^n=1,\; YXY^{-1}=X^r\rangle, \quad \text{ord}_p(r)=n. \ee
From the group presentation an element in $\mathbb Z[G]$ has a unique representation
$$ F(X,Y)= \sum_{i=0}^{p-1} \sum_{j=0}^{n-1} a_{ij} X^i Y^j= f_0(X)+f_1(X)Y+\cdots + f_{n-1}(X)Y^{n-1}, $$
with $f_j(x) =\sum_{i=0}^{p-1 }a_{ij}x^i$ in $\mathbb Z[x]$.

From the group relations, any group character $\chi$ must satisfy $\chi(Y)^n=1$ and $\chi(X)=1$. That is,
we have the $n$ characters with $\chi(X)=1$ and $\chi(Y)$ a $n$th root of unity, contributing the integer
\be \label{AformH} A=\prod_{y^n=1} F(1,y) \ee
to the group determinant. 
Notice, this is also the $\mathbb Z_n=\langle Y\; |\; Y^n=1\rangle$ circulant determinant 
for $F(1,Y)$, and so can be written
\be \label{detformA} A = \det \begin{pmatrix} f_0(1) & f_1(1) & \cdots  & f_{n-1}(1) \\
 f_{n-1}(1) & f_0(1) & \cdots  & f_{n-2}(1) \\
 \vdots  & \vdots  &  & \vdots \\
  f_1(1) & f_2(1) & \cdots & f_{0}(1)\end{pmatrix}. \ee

Writing $p-1=nt$ we also have $t$ degree $n$ representations. Writing $\om$ for a primitive $p$th root of unity,  it is readily checked that
\be       \rho(Y)=\begin{pmatrix} 0 & 1 & 0 &  \cdots & 0 \\ 0 & 0 & 1 & \cdots & 0 \\ \vdots  & \vdots  & \vdots &  & \vdots  \\ 0 & 0 & 0 & \cdots & 1 \\1 & 0 & 0 & \cdots  & 0 \end{pmatrix},\;\;\; \rho(X) =\left( \begin{matrix} \om & 0 & \cdots  & 0 \\ 0 & \om^r & \cdots &  0 \\ \vdots & \vdots &   & \vdots \\  0 & 0 & \cdots & \om^{r^{n-1}} \end{matrix}\right),\ee
satisfy the group relations.
Plainly
$$ \sum a_i\rho(X^i) = \left( \begin{matrix} \sum a_i\om^i & 0 & \cdots  & 0 \\ 0 & \sum a_i(\om^r)^i & \cdots & 0 \\ \vdots & \vdots &  & \vdots  \\  0 & 0 & \cdots & \sum a_i(\om^{r^{n-1}})^i \end{matrix}\right), $$
while multiplying by $\rho(Y^j)$ on the right just results in a cyclical shift of $j$ places to the right.
Hence
$\det(F(\rho(X),\rho(Y))=\det \left( \sum a_{ij}\rho(X^i Y^j)\right)$ becomes
\be \label{BformH}  B(\om)=\det \begin{pmatrix}
 f_0(\om) & f_1(\om) & \cdots  & f_{n-1}(\om) \\   
f_{n-1}(\om^r) & f_0(\om^r) & \cdots & f_{n-2}(\om^r) \\
\vdots & \vdots &   & \vdots \\
f_1(\om^{r^{n-1}}) & f_2(\om^{r^{n-1}}) & \cdots  & f_0(\om^{r^{n-1}}) \end{pmatrix}.
\ee
When $G=GA(1,p)$ we have only one degree $n=p-1$ representation, $B=B(\om)$ will be an integer, and
$$   D=A B^{p-1}. $$
When $t\geq 2,$ replacing $\om $ by $\om^j$, with $j$ running through representatives $k_1,\ldots ,k_t$
of the cosets $\mathbb Z_p^*/\langle r \rangle$, will give us the $t$ different degree $n$ representations.
Hence the integer group determinant for $F(X,Y)$  takes the form
\be \label{detform}  D=A B^n,\;\; \quad B=\prod_{i=1}^t B(\om^{k_i}). \ee
Notice that the $B(\om^j)$ need not be integers themselves, potentially lying in a degree $t$ 
extension (it is not hard to see that it will be fixed by the subgroup of cyclotomic automorphisms generated by $\om \mapsto \om^r$), but their product (unchanged by all the cyclotomic automorphisms $\om \mapsto \om^j$)  will be an integer.

\section{Divisibility conditions}

Observe that the integer $A$ in \eqref{detform} will satisfy the $\mathbb Z_n$ divisibility requirements \eqref{divy}, while
from \eqref{detformA} and \eqref{BformH} we have 
\be \label{bitcong} B(\om^j)\equiv A \bmod (1-\om), \ee
and hence, since $A$ and $B$ are integers with $B\equiv A^t \bmod (1-\om)$ and $\abs{1-\om}_p<1$, 
\be \label{intcong}  B\equiv A^t \bmod{p}. \ee
In particular, since $p$ cannot divide $A$ or $B$ without dividing both,
\be \label{divvy}  p\mid D \Rightarrow p^{n+1}\mid D. \ee
Obseve that for $G=GA(1,p)$ this is optimal;
\be \label{multp} F(X,Y) = 1+Y+ \cdots + Y^{p-1} \bmod{(Y^n-1)} + m h(X,Y) \ee
with
\be \label{defh} h(X,Y):=(1+X+\cdots +X^{p-1})(1+Y+\cdots +Y^{n-1}) \ee
has $A=p+mpn$, $B(\om)=p$ (the $\mathbb Z_n$ determinant with $m=0$), and $D=(1+mn)p^{p}$.
In Theorem \ref{generalt} we will show that \eqref{divvy} is always sharp.

For $t>1$, the $B(\om^j)$ are algebraic integers in the degree $t$ extension $\mathbb Q(\alpha)$,
$$ \alpha=\om+\om^r+\cdots + \om^{r^{n-1}}, $$
 fixed by  $\om\mapsto \om^r$.
Rather than \eqref{intcong}, it makes sense to stay in $\mathbb Q(\alpha)$ and use \eqref{bitcong} directly. That is, $B(\om)$ is an algebraic integer in $\mathbb Q (\alpha)$ with $\abs{B(\om)-A}_p<1$. Since the $\om^j$, $j=1,\ldots ,p-1,$ form an integral basis for $\mathbb Q (\om)$, it is not hard to see that the  conjugates of $\alpha$
$$ \alpha_i = \sum_{j=0}^{n-1}\om^{k_ir^j}, \quad i=1,\ldots ,t, $$
form an integral basis for for $\mathbb Q(\alpha)$. Observing that the $\abs{\alpha_i-n}_p<1$, $i=1,\ldots ,t,$   with  $\sum_{i=1}^t (\alpha_i-n)= -p$, it is readily seen that \eqref{bitcong}
becomes
\be \label{withp}  B(\om^j) = A+ \beta_0 p + \sum_{i=1}^{t-1} \beta_i (\alpha_i-n), \quad \text{ for some }\beta_0,\ldots ,\beta_{t-1}\in \mathbb Z, \ee
or equivalently
\be \label{withoutp}  B(\om^j) = A+ \sum_{i=1}^{t} \beta_i (\alpha_i-n), \quad \text{ for some } \beta_1,\ldots ,\beta_t\in \mathbb Z.\ee
For a $\gamma$ in $\mathbb Q(\alpha)$ we write $N(\gamma)$ for the norm from $\mathbb Q(\alpha)$ to $\mathbb Q$
$$ N(\gamma) = \prod_{i-1}^t \sigma_i(\gamma),\quad \sigma_i(\om)=\om^{k_i},\;\;  i=1,\ldots ,t,$$
and immediately obtain the first part of the  following theorem. 

\begin{theorem} \label{generalt}
The integer group determinants for $\mathbb Z_p \rtimes \mathbb Z_n$, $p-1=nt$, must be of 
the form
\be \label{genthm}  m \:N\left(m+ \sum_{i=1}^{t} \beta_i (\alpha_i-n)\right)^n, \quad \beta_1,\ldots ,\beta_t\in \mathbb Z, \ee
where $m$ is a $\mathbb Z_n$ integer group determinant.

We achieve all such values with $m$ coprime to $n$.
\end{theorem}

\begin{proof} Newman showed that every $m$ coprime to $n$ is a $\mathbb Z_n$ determinant. We show in Lemma \ref{ex} below that we can achieve all \eqref{genthm} with $\gcd(n,m)=1$. 
\end{proof}

Notice that $\gamma=\prod_{i=0}^{n-1} (1-\om^{r^i})$ is an algebraic integer in $\mathbb Q (\alpha)$ with $\abs{\gamma-kp}_p<1.$ Hence we can achieve any $kp^{n+1}$ with $\gcd(k,n)=1$, by  taking $A=kp$, $B(\om)=\gamma$, and  $B=N(\gamma)=p$. In particular \eqref{divvy}  is sharp.

Notice that $F(X,Y)=-Y$ gives us $D=-1$ and so, by multiplicativity \eqref{mult}, 
all these groups have the nice property that $\pm m$ is an integer group determinant whenever $m$ is. This is certainly not true for all groups.

Of course for  general  $n$ we do not even know the  values of $A$ (since we do not know the integer group determinants for $\mathbb Z_n$).
For example, for $GA(1,p)$, $p>5$,  we might want
 $p$ to be a `safe' prime, that is $n=2k+1$ with $k$ a Sophie Germain prime, if we want to obtain a complete description of the integer group determinants.

\section{Values that we can achieve}
With $h(X,Y)$ as in \eqref{defh}, we shall make frequent use of shifts of elements $G(X,Y)\in \mathbb Z[G]$ 
of the form
\be \label{shifty}  F(X,Y)= G(X,Y) +t(X) (1+\cdots + Y^{n-1}) + mh(X,Y),  \ee
to obtain families of integer group determinants related to the value of the determinant 
obtained for  a particular $G(X,Y)$.

Notice that $F(1,y)=G(1,y)+ (t(1)+pm)(1+y+\cdots + y^{n-1})$ and 
$$ A=\left( G(1,1)+nt(1)+mnp \right) \prod_{y^n=1,y\neq 1} G(1,y), $$
while $B(\om)$ for $F$  will be the determinant of the matrix for $G$ with $t(\om^{r^{i-1}})$ added to every element in the $i$th row. 
 Notice, subtracting the last column from the other columns, this determinant will be linear in the $t(\om^{r^i})$, and the $B(\om)$ for $F$ takes the form
\be \label{shiftedB}  B_F(\om) = B_G(\om) + \alpha(\om)t(\om) + \alpha(\om^r)t(\om^r)+\cdots + \alpha(\om^{r^{q-1}})t(\om^{r^{q-1}}), \ee
where $B_G(F)$ is the $B(\om)$ determinant for $G$ and $\alpha(\om)$ the determinant
when we replace the first row in that $G$ matrix by all 1's, the other coefficients just being the appropriate conjugate $\om \mapsto \om^{r^{j}}$.

We  can achieve all the values coprime to $n$ satisfying \eqref{genthm} with one family of $G$. 

\begin{lemma} \label{ex} Suppose that $s<n$ has $\gcd(s,n)=1$ and 
$$ G(X,Y)=1+\cdots + Y^{s-1}  $$
then $F(X,Y)$ in \eqref{shifty}
has
$$ A=s+nt(1) +mnp, \;\; \quad B(\om)=s+ \sum_{j=0}^{n-1} 
t(\om^{r^j}). $$
Taking  $t(x)=c+ \sum_{i=1}^t\beta_i(1-x^{k_i})$ we get
$$ A=s+cn+mnp,\;\; B(\om)=s+cn+ \sum_{i=1}^t \beta_i( n-\alpha_i). $$
For example, when $n=p-1$ taking $t(x)=c+b(1-x)$ gives
$$ A=s+cn+mnp,\;\; B=s+cn+bp. $$
and when $n=(p-1)/2$ taking  $t(x)=c+a(1-x^u)+b(1-x^v),$ with $u$ a quadratic 
residue and $v$ a quadratic non-residue mod $p$, gives
$$ A=s+cn+mnp,\;\; B(\om)=s+cn+a(p-\sqrt{\ve p})/2 +b(p+\sqrt{\ve p})/2. $$ 
\end{lemma}

\begin{proof}[Proof of Lemma \ref{ex}]
Since $\gcd(n,s)=1$ we have 
$$\prod_{y^n=1,y\neq 1} (1+y+\cdots +y^{s-1})=1,\quad G(1,1)=s,\quad A=s+nt(1) +mnp, $$
and $B_G(\om)=s$, the $\mathbb Z_n$ determinant of $1+y+\cdots y^{s-1}$.

Subtracting the first column from the remaining $(n-1)$ columns and expanding along the first row, we see that
$$\alpha(\om)=\det \begin{pmatrix} M_1 \\ M_2 \end{pmatrix}$$
where $M_1$ and $M_2$ are the $(n-s) \times (n-1)$ and $(s-1)\times (n-1)$ matrices
$$ M_1 = \begin{pmatrix}  1 & 1 & \cdots & 0 & 0 \\
 0 & 1 & \cdots & 0 & 0 \\
 \vdots  & & & \vdots\\
  0 & 0 &  \cdots & 1 & 1
\end{pmatrix},\quad M_2 = \begin{pmatrix} -1 & -1 & -1 & \cdots & 0 & 0 \\
0 & -1 & -1 & \cdots & 0 & 0 \\
 \vdots & & & & \vdots\\
 0 & 0 & 0 &  \cdots & -1 & 0
\end{pmatrix}.$$
That is, $\alpha(\om)$ will be the resultant of the polynomials $1+x+x^2+\cdots +x^{s-1} =(x^s-1)/(x-1)$ and
$-x-\cdots - x^{n-s}=-x(x^{n-s} -1)/(x-1)$. Plainly
$$\Res\left(\frac{x^s-1}{x-1},-x\right)=\prod_{x^s=1,x\neq 1} -x=1,$$ 
while from classical results on the resultant of two cyclotomic polynomials \cite{Lehmer} (or \cite{Apostol}) we know that 
$$\Res\left(\frac{x^s-1}{x-1},\frac{x^{n-s}-1}{x-1}\right)=1$$ 
if  $\gcd(s,n)=1$  (and zero otherwise).  Hence $\alpha(\om)=1$. 

When $n=p-1$ we know that $r$ is a primitive root mod $p$, the $r^i$ run through all the values except 0 mod $p$ and $\sum_{i=0}^{n-1} \om^{r^i}=-1$.

When $n=(p-1)/2$ the $ur^i$ run through the quadratic residues mod $p$, the $vr^i$
the quadratic nonresidues and from the classical evaluation of quadratic Gauss sums (eg Lidl \& Niederreiter \cite[p199]{Lidl})
$$ \sum_{i=0}^{n-1} \om^{ur^i}=\frac{1}{2}(-1+\sqrt{\ve p}),\quad \sum_{i=0}^{n-1} \om^{vr^i}=\frac{1}{2}(-1-\sqrt{\ve p}), \quad \ve=\begin{cases} +1 & \text{ if $p=1 \bmod 4$,}\\ -1 & \text{ if $p\equiv 3\bmod 4.$} \end{cases} $$
\end{proof}

\section{ The general affine groups $GA(1,p)$}

 For $n=p-1$ we have the congruence condition:
\be \label{GAcong}  D=AB^{p-1}, \quad B\equiv A \bmod p, \ee
where $A$ must satisfy the divisibility conditions \eqref{divy} for $\mathbb Z_n$.
Lemma \ref{ex} shows that this is if and only if for the values coprime to $n.$
We are also able to obtain  all the multiples of $n^2$ satisfying this condition,
giving us  the analog of \eqref{cyclic} in this case. 

\begin{theorem}\label{GA}
For $GA(1,p)$ the  integer group determinant values contain all the
$$m(m+\ell p)^{p-1}, \quad \gcd(m,n)=1 \text{ or } n^2\mid m.$$
\end{theorem}

This at least us  gives all the values coprime to $n$.
In particular, the multiples of $p$ coprime to $n$ are exactly the $mp^p$, $\gcd(m,n)=1$
(we can take $A=mp$, $B=p$). All the $m\equiv \pm 1 \mod p$ with $\gcd(m,n)=1$ 
are achieved (with  $A=m$, $B=\pm 1$), but the remaining $p\nmid m$ will need to contain a nontrivial $(p-1)$'st power.

For general $n$ we can not hope to say much  more; the $A$ is a $\mathbb Z_n$ determinant and these values are only known for special cases of $n$.

Note,  Theorem \ref{GA} already gives a complete description for $GA(1,5)$.
When $q$ is an odd Sophie Germain prime and $n=2q$, $p=2q+1$ 
we might also hope to say more. In that case the divisibility condition becomes 
$$2\mid A\Rightarrow 2^2 \mid A,\quad q\mid A\Rightarrow q^2 \mid A. $$
That is, we just need to deal with the cases of even $A$ with $q\nmid A$ and odd $A$ with $q\mid A$. We give examples of this in Section 5.3 below.


\begin{proof}[Proof of Theorem \ref{GA}] From  Lemma \ref{ex} we can achieve all $A=m$
with $\gcd(m,n)=1$ and any $B\equiv A$ mod $p$.
We achieve the $A=m$ with $n^2\mid m$ and any $B\equiv A$ mod $p$  from the following lemma.
\end{proof}

\begin{lemma} If $F(X,Y)$ is of the form \eqref{shifty} with
$G(X,Y)=1-YX $, then
$$ A=n^2( t(1) +mp), \;\; B(\om)= \sum_{j=0}^{n-1} \alpha(\om^{r^j})t(\om^{r^j}) $$
where
$$  \alpha(\om)=\sum_{i=0}^{n-1} \om^{-(r^i-1)/(r-1)}. $$
If $n=p-1$ and $k(r-1)\equiv -1$ mod $p$ then  $t(x)=c+a(1-x^k)$ has
$$ A=n^2(c+mp) \equiv c \bmod p,\quad   B=c+ap. $$

\end{lemma}

\begin{proof} We have $G(1,y)=1-y$ and $A$ is clear.
Writing $A_i=\om^{r^i}$, we have
\begin{align*}  B_G(F) & = \det  \begin{pmatrix} 1 & -A_0 & 0 & \cdots  &  0& 0 \\
     0 & 1 & -A_1 & \ldots & 0 & 0\\
\vdots &  &   & &   & \vdots \\  0 & 0 & 0 &   \cdots & 1 & -A_{n-2} \\ -A_{n-1} & 0 &  0 & \cdots & 0 & 1 \end{pmatrix} \\
& = \det \begin{pmatrix} 1 & -A_0 & 0 & \ldots  & 0 \\
     0 & 1 & -A_1 & \ldots & 0 \\
\vdots &  &   &   & \vdots \\  0 & 0 &   & \cdots  & 1 \end{pmatrix} +
(-1)^n A_{n-1}\det \begin{pmatrix}  -A_0 & 0 & \ldots  & 0 \\
     1 & -A_1 & \ldots & 0 \\
\vdots &     &   & \vdots \\  0 &  0 & \cdots &  -A_{n-2} \end{pmatrix} 
\\= & 1-A_0A_1\cdots A_{n-1}=1- \om^{(r^n-1)/(r-1)}=0, 
\end{align*}
and the coefficient of $t(\om)$ will be of the form
\begin{align*}  \alpha (\om)=& \det  \begin{pmatrix} 1 & 1 & 1  & \ldots & 1 & 1 \\
     0 & 1 & -A_1 & \ldots & 0 &  0 \\
\vdots &  &   & &  & \vdots \\  0 & 0 & 0 &  \cdots & 1 & -A_{n-2} \\ -A_{n-1} & 0 & 0 &  \cdots & 0 & 1 \end{pmatrix}\\
& =1 + (-1)^n A_{n-1}\det  \begin{pmatrix}  1 & 1  & \ldots & 1 & 1 \\
      1 & -A_1 & \ldots & 0 &  0 \\
\vdots &    & &  & \vdots \\   0 & 0 &  \cdots & 1 & -A_{n-2} \end{pmatrix}
\\
& =1 + A_{n-1}\left( 1 + (-1)^{n-1}A_{n-2}\det  \begin{pmatrix}  1 & 1  & \ldots & 1 & 1 \\
      1 & -A_1 & \ldots & 0 &  0 \\
\vdots &    & &  & \vdots \\   0 & 0 &  \cdots & 1 & -A_{n-3} \end{pmatrix}\right)
\\& =  1+A_{n-1}+A_{n-1}A_{n-2}+A_{n-1}A_{n-1}A_{n-3}+\cdots + A_{n-1}A_{n-2}\cdots A_{1}, \end{align*}
with
$$ A_{n-1}\cdots A_i= \om^{r^i+\cdots + r^{n-1}}= \om^{-1-r-\cdots r^{i-1}}= \om^{k(r^i-1)}. $$
When $n=p-1,$ $r$ is a primitive root mod $p$ and the $k(r^i-1)$, $i=0,...,n-1$ run through all the values mod $p$ except for $-k$ and we get $\alpha (\om)=-\om^{-k}$. Hence $t(x)=c+a(1-x^k)$ has
$$ \alpha(\om)t(\om)+\cdots + \alpha (\om^{r^{j-1}})t(\om^{r^{j-1}}) = \sum_{i=0}^{n-1}  -(c+a)\om^{-kr^i}+a=c+a+na=c+pa. $$
\end{proof} 

\subsection{SmallGroup(20,3) or GA(1,5)}
We have
\be \label{GpPres} GA(1,5)=\langle X, Y \; | \; X^5=Y^4=1,\; YXY^{-1}=X^2\rangle, \ee
and for an element in $\mathbb Z [G]$,
$$ F(X,Y)= \sum_{i=0}^4 \sum_{j=0}^3 a_{ij} X^iY^j= f_0(X)+f_1(X)Y+f_2(X)Y^2+f_3(X) Y^3$$
with $f_j(x) =\sum_{i=0}^4 a_{ij}x^i$ in $\mathbb Z[x]$,
our integer group determinants take the form
\be \label{formH} D=AB^4, \ee
where $A$ and $B$ are the integers
$$ A=F(1,1)F(1,-1)F(1,i)F(1,-i)$$
and
\be \label{Bform20}  B=\det \begin{pmatrix}
 f_0(\om) & f_1(\om) & f_2(\om) & f_3(\om) \\   
f_3(\om^2) & f_0(\om^2) & f_1(\om^2) & f_2(\om^2) \\
f_2(\om^4) & f_3(\om^4) & f_0(\om^4) & f_1(\om^4) \\
f_1(\om^3) & f_2(\om^3) & f_3(\om^3) & f_0(\om^3) \end{pmatrix}.
\ee
These must satisfy
\be \label{congH}  B\equiv A \bmod 5, \ee
and, since $A$ is a $\mathbb Z_4$ determinant
\be  \label{divH} 2 \mid A \;\Rightarrow \; 2^4\mid A. \ee
Notice, this says that  $A$ and $B$ are either both divisible by 5 or both coprime to 5, and $A$ and $B^4$ are either odd or a multiple of $2^4$, immediately giving us the divisibility restrictions:
\be \label{divs} 5\mid D \; \Rightarrow  \; 5^5\mid D,\hskip0.3in 2\mid D \; \Rightarrow \; 2^4\mid D. \ee

From Theorem \ref{GA} these conditions 
\eqref{formH}, \eqref{congH} and \eqref{divH} are also sufficient.

\begin{theorem} The integer group determinants for SmallGroup(20,3) are exactly the integers of the form 
\be \label{20form}  m(m+5\ell)^4,\;\; \text{ $m$ odd or $2^4\mid m$.} \ee
\end{theorem}

\vskip1ex
\noindent
That is, the values coprime to 10 are the integers $\pm 1 \bmod 10,$ plus the integers $\pm 3$ mod 10 of the form $(\pm 3+10m)(3+10k)^4$.

\vskip1ex
\noindent
The odd multiples of 5 are  all the odd multiples of $5^5$.

\vskip1ex
\noindent
The multiples of 10 are all the multiples of $2^4\cdot 5^5$.

\vskip1ex
\noindent
The even determinants coprime to 5 consist of  all the $2^4m$ with $m\equiv \pm 1 \bmod 5$, plus the $2^4m$, $m\equiv \pm 2 \bmod 5$ of the form

$2^4 m,$ $m\equiv \pm 3 \bmod 10,$

$2^5(\pm 1 +10m)(3 +10k)^4,$

$2^6(\pm 3 +10m)(3 +10k)^4,$

$2^7(\pm 1+10m)(3+10k)^4, $

$2^8m, \; m\equiv \pm 2 \bmod 5.$


\begin{proof}  From   \eqref{formH}, \eqref{congH} and \eqref{divH}, we know that the determinants take the form
$$ D=AB^4,\quad B\equiv A \bmod 5, \quad 2\mid A \Rightarrow 2^4\mid A. $$
All these were obtained in Theorem \ref{GA} and \eqref{20form} is plain.
\end{proof}

It only remains to justify that these values must be of the stated form. We know from \eqref{divs} that the multiples of $5$ must be of the form $5^5t$ with $t$ odd or $16\mid t$ and we achieve all these with $A=5t$, $B=5$. Indeed with the $\pm$ sign, the odd multiples were all obtained from \eqref{multp}.

We achieve all the odd values coprime to 5 that are $t\equiv \pm 1 \bmod 5$ by taking $A=t$ and $B=\pm 1$.
For the odd values $t\equiv \pm 2 \bmod 5$ we must have $B^4\equiv 1 \bmod 5$, $A\equiv \pm 3\bmod 5$ and $\pm B\equiv 3 \bmod 5$. That is $A=\pm 3 +10m$ and $\pm B=(3+10k),$ with all these obtainable.

From \eqref{divs} we know that the even values coprime to 5 must be of the form $2^4t$
with $5\nmid t$. We obtain all such values with $t\equiv \pm 1 \bmod 5$  by taking $A=2^4t$ and $B=\pm 1$ and the  $t\equiv \pm 2 \bmod 5$ with
$t$ odd or $2^4\mid t$ by taking $A=t$, $B=\pm 2$. 

This just leaves the values $2^s t$, with $t$ odd and $s=5,6$ or 7
that are $\pm 2 \bmod 5$.
 Since $4\nmid s$ these must have
$A$ even and $2^4\mid A$. Since $s<8$, we must have $B$ odd and $2^s\parallel  A$.
When $s=5$ or 7 these values must have $A=2^sA_1$, $A_1\equiv \pm 1 \bmod 5$ odd and
$\pm B\equiv 3\bmod 5$ odd. When $k=6$ we must have $A=2^6A_1$ with $A_1\equiv \pm 3\bmod 5$ odd and $\pm B\equiv 3\bmod 5$ odd. All these are achieved in Theorem \ref{GA}.

\subsection{Generating all $B\equiv A\pmod p$ for general $GA(1,p)$}
Suppose that $A_0$ is a $\mathbb Z_{p-1}$ determinant corresponding to some $f(Y);$
$$ A_0 = f(1) A_1,\quad A_1:=\prod_{y^n=1,y\neq 1} f(y). $$ 
For this stem value we try to construct $F(X,Y)$ giving  related $A$ and all $B\equiv A\pmod p$. We first select  a  
$$G(X,Y)=\sum_{j=0}^{n-1} f_j(X)Y^j,  $$
with $G(1,Y)=f(Y)$, 
 and calculate
\begin{align*} \alpha(x) & = \det \begin{pmatrix} 1 & 1 & 1 & \cdots & 1 \\
f_{p-1}(x^r) & f_0(x^r) & f_1(x^r) & \cdots & f_{p-2}(x^r) \\
\vdots & \vdots & \vdots & & \vdots \\
 f_1(x^{r^{p-2}}) & f_2(x^{r^{p-2}}) & f_3(x^{r^{p-1}}) & \cdots & f_0(x^{r^{p-2}})
\end{pmatrix}  \\
 & = a_0+a_1x+\cdots + a_{p-1}x^{p-1} \pmod {x^p-1},
\end{align*}
the determinant for $B_G$ with first row replaced by 1's. Our goal is to find a $t(x)$ with
\be  \label{conditionst} t(1)=0 \quad \text{ and } \quad t(x)\alpha(x)=1+c_1x+\cdots +c_{p-1}x^{p-1}\pmod{ x^p-1}. \ee
If,  by chance, $\alpha(x)$ has two coefficients $a_I,a_J$ with $a_J-a_I=1$ then
$$t(x)=(x^{p-J}-x^{p-I})$$ 
will satisfy \eqref{conditionst}. Otherwise we calculate
$$ b(x)= (x-1)\alpha(x)= b_0+b_1x+\cdots +b_{p-1}x^{p-1} \bmod x^p-1. $$
If the coefficients $b_i$ have a common factor then we try again with a new $G(x,y)$.
Otherwise we find integers  $\lambda_i$ with $\lambda_0b_0+\cdots +\lambda_{p-1}b_{p-1}=1$ and we can take 
$$t(x)=(x-1)\sum_{j=0}^{p-1} \lambda_j x^{p-j} \pmod {x^p-1}. $$

\begin{lemma}\label{multlemma} Suppose that $s$ in $\mathbb N$ has $\gcd(s,p-1)=1$ then
$$ F(X,Y)=G(X,Y)\left(\frac{Y^s-1}{Y-1}\right) +(\ell + mt(X))\left(\frac{Y^n-1}{Y-1}\right) \pmod {Y^n-1} $$
has
$$ A_F=    sA_0+A_1\ell (p-1), \quad        B_F=sB_G+(p\alpha(0)-A_1)\ell +mp.   $$
In particular, by choice of $m$ we can achieve all $B_F\equiv A_F\bmod p$.
\end{lemma}
Notice, the previous results for $\gcd(A,p-1)$ and $n^2\mid A$ are achieved this way with $G(X,Y)=1$ and $G(X,Y)=1-YX$.

\begin{proof}
Observe that the coefficient of $Y^j$ in $\tilde{G}(X,Y)=G(X,Y)(1+Y+\cdots + Y^{s-1})$ consists of the sum of $s$ consecutive $f_{j-i}(X)$ and the matrix for $B_{\tilde{G}}$ is the matrix for $B_G$ multiplied
on the right by a circulant matrix, $C$ say, for $1+Y+\cdots +Y^{s-1}$ whose determinant is $s,$ since $\gcd(s,p-1)=1$. Hence $B_{\tilde{G}}=sB_G$. Clearly if we multiply the matrix for $\alpha(x)$ for $G$ on the right by $C$ we get the matrix for $\alpha(x)$ for $\tilde{G}$  but with the first row consisting of $s$'s instead of $1$'s. Hence the $\alpha (x)$ for  $\tilde{G}$ is the same as the $\alpha(x)$ for $G$. Therefore
$$ F(X,Y)=\tilde{G}(X,Y)+h(X)(1+Y+\cdots +Y^{p-2}) $$
has 
$$B_F=B_{\tilde{G}}+ h(\om)\alpha(\om)+h(\om^r)\alpha(\om^r)+ \cdots + h(\om^{r^{p-2}})\alpha(\om^{r^{p-2}}), $$
and if $h(x)\alpha(x)=u(0)+u(1)x+ \cdots +u(p-1)x^{p-1} \mod {x^p-1},$  we get
$$ B_F=sB_G+pu(0)-h(1)\alpha(1). $$
By observing that $y=1$ reduces the $B$ to an $A,$ we get that $\alpha(1)=A_1$.
The result follows since $u(0)=\ell\alpha (0)+m$, $h(1)=\ell$.
\end{proof}

\subsection{GA(1,p) for the safe primes $p=7,11,23,\ldots $.} 

\vspace{2ex}
Suppose that $p=2q+1,$ where $q$ is an odd Sophie Germain prime.
$$ GA(1,p)=\langle X,Y \; : \; X^p=Y^{2q}=1, \;\;YX=X^rY\rangle, \quad \text{ ord}_p(r)=p-1. $$ From Lacquer \cite{Laquer} we know that the $\mathbb Z_{2q}$ integer determinants are the $A=m$ with $\gcd(m,2q)=1$, the $A=4q^2m$ any $m\in \mathbb Z$, the $A=4k$ with $\gcd(k,q)=1$, and the $A=q^2(2\ell+1)$.
By Theorem \ref{GA} we obtain all
$$ D=AB^{2q},\quad B\equiv A \bmod p, $$
for any $A$ of the first two types. We try to show this for the other two forms.
We need only two polynomials.

From \cite{Laquer} we get $A_0=q^2$ using $f(y)=(1+y+\cdots + y^q)-y$, $A_1=f(1)=q$.  Hence, if we can find a
$G(x,y)$ with $G(1,y)=f(y)$ and a $t(x)$ satisfying \eqref{conditionst}, then  taking $s=1$ in Lemma \ref{multlemma} gives
$$ A=q^2(1+2\ell),\quad B=B_G+(p\alpha(0)-q)\ell+mp. $$

Similarly we get $A_0=2^2$ from $f(y)=1+y^2,$ $A_1=f(1)=2$,  and just have to find a $G(x,y)$ with
$G(1,y)=1+y^2$ and a $t(x)$ satisfying \eqref{conditionst}. Hence Lemma \ref{multlemma} gives
$$ A= 2^2(s+q\ell),\quad B=sB_G+(p\alpha(0)-2)\ell+mp, $$
where to get $A=2^2k$, $\gcd(k,q)=1$ with $\gcd(s,2q)=1$ we take $s=k$, $\ell=0$
when $k$ is odd and $s=k+q,$ $\ell=-1$ when $k$ is even. Since $F\mapsto FY$ sends $(A,B)\mapsto (-A,-B)$ we can assume that $k>0$.

This algorithm  succeeded for both values of  $A_0$  when  $p=7,11,23$ as shown in the tables below:

\vskip2ex
\nni
{$p=7=2\cdot 3+1$, using $r=3.$}

\vskip2ex
\nni
$\begin{array}{|c|c|c|c|c|}\hline
A_0 & G(x,y) & \alpha(x)  & t(x) & B_G \\ \hline
2^2 & 1+(1-x)y+y^2 & x-x^4+2x^6 & 1-x^3  & -3 \\
3^2 & 1+xy^2+y^3 & 1+x^5+x^6 & 1-x^3 & 2 \\ \hline
\end{array}$

\vskip5ex
\nni
{$p=11=2\cdot 5+1,$ using $r=2.$}

\vskip2ex
\nni
$\begin{array}{|l|p{4.5cm}|p{5cm}|c|c|}\hline
A_0 &\hspace{1.3cm} $G(x,y)$ & \hspace{3cm} $\alpha(x)$  & t(x) & B_G \\ \hline
2^2 & $1+(x-x^2)y+y^2$ &  $2x-2x^2-3x^3-3x^4-x^5+5x^7+4x^8+3x^9-3x^{10}$  & x^3-x^2 & -7 \\ \hline
5^2 & $1+(1-x)y +y^2+y^3+y^4+y^5$ & $4+2x-x^4+2x^5+x^6-2x^7-x^{10}$ & x^2-x & 25 \\ \hline
\end{array}$

\vskip3ex
\nni
{$p=23=2\cdot 11+1$, using $r=5.$}

\vskip2ex
\nni
$\begin{array}{|l|p{4cm}|p{5cm}| c|c|}\hline
A_0 & \hspace{1.2cm} $G(x,y)$ &\hspace{2.5cm} $\alpha(x)$  &  t(x) & B_G \\ \hline
2^2 & $1+(1-x)y + y^2$ & $-131+37x+53x^2+82x^3-31x^4-105x^5-120x^6+24x^7+3x^8+111x^9+105x^{10}
-47x^{11}-8x^{12}-34x^{13}-47x^{14}+7x^{15}+57x^{16}+13x^{17}+12x^{18}
-20x^{19} +24x^{20}+62x^{21}-45x^{22}$ &   x^6-x^5  & -364 \\ \hline
11^2 & $1+(1-x)y+y^2+y^3+y^4+y^5+y^6+y^7+y^8+y^9+y^{10}+y^{11}$ & $1+2x+3x^2-x^3+3x^4-x^7+x^9-x^{10}+x^{12}+x^{13}-x^{14}+x^{18}+3x^{19}-x^{20}-x^{21}+x^{22} $ &   1-x^6 & 6 \\ \hline

\end{array}$


\vskip3ex
\nni
Hence we have the following.

\begin{theorem} \label{p=SG} For $q=3,5$ or $11$
the integer group determinants for $GA(1,p)$, $p=2q+1$,  are exactly the integers of the form
$$ m(m+p\ell)^{p-1},\quad \text{ $m$ odd or } 4\mid m \text{ with }\;\; q\nmid m \text{ or } q^2\mid m. $$
\end{theorem}

ChatGPT 5.6 claims  that the same algorithm works for $p=47,59$ and 83
(and, using its own reformulation, that the theorem holds for  the next five safe  primes!).

\section{The case $n=(p-1)/2$}\label{halfer}
When $n=(p-1)/2$  the values $r^j$ run through the quadratic residues mod $p$ and
$$ B=B(\om) B(\om ^u) $$
where $u$ is a quadratic non-residue mod $p$. Note, for $p\equiv 3$ mod 4  we can take $u=-1$ and $B=|B(\om)|^2$. Moreover, from
Gauss sums we get that
$$ \sum_{j=0}^{n-1} \om^{r^j} = \frac{1}{2}\left(-1+\sqrt{\ve p}\right),\;\; \ve=\begin{cases} 1, & \text{ if $p\equiv 1$ mod 4,} \\ -1, & \text{if $p\equiv 3$ mod 4.} \end{cases}$$
Hence the value of $B(\om )$ is an algebraic integer in $\mathbb Q (\sqrt{\ve p})$
and $B$ will be a norm
$$ N\left( \alpha + \frac{1}{2} (p+\sqrt{\ve p})\beta \right)=\left(\alpha+\frac{1}{2}\beta p\right)^2 - \frac{\ve p}{4} \beta^2,  \quad \alpha,\beta \in \mathbb Z .$$
Observing  that $|1-\om|_p<1$, it makes sense to use \eqref{bitcong} to replace the integer congruence $B\equiv A^2 \bmod p$ of \eqref{intcong}  by a more precise statement about $B(\om)$ itself:
\be \label{quadcong} B(\om) = A  + \alpha p + \frac{1}{2}(p+\sqrt{\ve p})\beta, \quad \alpha,\beta \in \mathbb Z.\ee
Lemma \ref{ex} immediately gives us  all such values with $A$ coprime to $q$.

\begin{theorem}\label{half}
For $n=(p-1)/2$ the integer group determinant values with $A$ coprime to $n$ are exactly the 
$$m\; N\left(m+\alpha p+ \frac{1}{2}(p+\sqrt{\ve p})\beta \right)^{n}, \quad \gcd(m,n)=1. $$
\end{theorem}
Notice, the multiples of $p$ coprime to $n$ are exactly the $mp^{(p+1)/2}$, $\gcd(m,n)=1,$
since we can take $A=mp$ and $B(\om)=\sqrt{\ve p},$ and \eqref{divvy} is optimal. 
We can achieve all the $m\equiv \pm 1 \mod p$ with $\gcd(m,n)=1$ with $B(\om)=\pm 1$.


\subsection{SmallGroup(21,1) or $\mathbb Z_7 \rtimes \mathbb Z_3$}
We have 
$$ \text{SmallGroup}(21,1)=\langle X,Y \; : \; X^7=Y^3=1, YX=X^2Y\rangle, $$
and for an element 
$$ F(X,Y)=\sum_{0\leq j\leq 2} f_j(X)Y^j,\quad f_j(X)=\sum_{0\leq i\leq 6} a_{ij}X^i,$$
we have the group determinant
$$ D=AB^3, \quad B=B(\om)B(\om^{-1})=\abs{B(\om)}^2, $$
where
$$ A= \prod_{y^3=1} F(1,y),\quad B(\om )=\det\begin{pmatrix} f_0(\om) & f_1(\om) & f_2(\om) \\  f_2(\om^2) & f_0(\om^2) & f_1(\om^2) \\ f_1(\om^4) & f_2(\om^4) & f_0(\om^4) \end{pmatrix}, \quad \om=e^{2\pi i/7}. $$
Notice that $B(\om)$ lies in $\mathbb Z[\om+\om^2+\om^4]=\mathbb Z [(-1+\sqrt{7}i)/2]$.

From \eqref{divy} and \eqref{quadcong} we have  the restrictions
\be \label{Adiv} 3\mid A \;\;\Rightarrow \;\; 3^2\mid A, \ee
and 
\be \label{quadcong7} 
 B(\om) = A + 7\alpha + \frac{1}{2}\left( 7+ i\sqrt{7} \right) \beta,  \quad \alpha,\beta\in \mathbb Z.\ee

These conditions \eqref{Adiv} and \eqref{quadcong} are if and only if.

\begin{theorem} The integer group determinants for SmallGroup(21,1) are exactly the integers of the form
\be \label{21form}   m  \; N\left(m+7\ell + \frac{1}{2}(7+\sqrt{7}i) \beta \right)^3,\quad 3\nmid m \text{ or } 9\mid m. \ee
\end{theorem}

That is, the integer group determinants for SmallGroup(21,1) which are divisible by 7 are exactly the $7^4m$ with $3\nmid m$ or $3^2\mid m$.

The integer determinants $m\equiv \pm 1 \bmod 7$ are exactly those with $3\nmid m$ or $3^2\mid m$.

The integer determinants $\pm 2$ or $\pm 3 \bmod 7$  are exactly the integers of the form
$$ m\; N\left(m + 7\alpha + \frac{1}{2}\left( 7+ i\sqrt{7} \right) \beta\right)^3,\quad 3\nmid m \text{ or } 3^2\mid m, $$
for some  $m\equiv \pm 2,\pm 3 \bmod 7,$ and  $\alpha,\beta \in \mathbb Z.$


\begin{proof} From Theorem \ref{half} we are just left to obtain the values \eqref{21form}
where $m$ is a multiple of 9. We take \eqref{shifty} with 
$$ G(X,Y)= (X+X^2-1)-Y, \quad t(x) = c+ a(X^5-X^3)+b(X^6-X^3), \quad m=0, $$
so $G(1,y)=1-y$ and $A=3c\cdot 3=9c$. We have
$$ B_G=\det \begin{pmatrix} \om +\om^2-1 & -1 & 0 \\
0 & \om^2 +\om^4-1 & -1 \\
-1 & 0 & \om^4+\om -1 \end{pmatrix} = 2\sqrt{7}i $$
and
$$\alpha (\om) = \det\begin{pmatrix} 1 & 1 & 1 \\
0 & \om^2 +\om^4-1 & -1 \\
-1 & 0 & \om^4+\om -1 \end{pmatrix} =-2\om^4-\om^2-\om. $$
So $\om^3\alpha(\om)=-2-\om^5-\om^4$, $\om^5\alpha(\om)=-2\om^2-1-\om^6$, $\om^6\alpha(\om)=-2\om^3-\om -1,$ with
$$\sum_{0\leq j\leq 2} \om^{k2^j}=\frac{1}{2}(-1 \pm \sqrt{7}i), $$ 
the $+$ sign for the squares $k=1,2,4$ mod 7 and $-$ sign  for the non-squares $k=3,5,6$, and
\begin{align*} \sum_{0\leq j\leq 2} \alpha (\om^{2^j}) &=2-2\sqrt{7}i, \quad \sum_{0\leq j\leq 2} \om^{3\cdot 2^j}\alpha (\om^{2^j})=-5, \\ \sum_{0\leq j\leq 2}\om^{5\cdot 2^j}\alpha (\om^{2^j}) & =(-3-\sqrt{7}i)/2, \quad  \sum_{0\leq j\leq 2}\om^{6\cdot 2^j}\alpha(\om^{2^j})=(-3+\sqrt{7}i)/2
\end{align*}
and  
$$B(\om) = 2c + 2(1-c)\sqrt{7}i+ \frac{1}{2}(7-\sqrt{7}i)a + \frac{1}{2}(7+\sqrt{7}i)b. $$
By choice of $a,b$ we can obtain any $A=9c$, $B(\om)= \left(9c+7\alpha + \frac{1}{2}(7+i\sqrt{7})\beta \right)$.
\end{proof}

\subsection{ SmallGroup(55,1) or $\mathbb Z_{11}\rtimes \mathbb Z_5$}
We have 
$$ \text{SmallGroup}(55,1)=\langle X,Y \; : \; X^{11}=Y^5=1, YX=X^4Y\rangle, $$
and
$$ D=AB^5,    \quad 5\nmid A \text{ or } 5^2\mid A, $$
where $B=N(B(\om))=|B(\om)|^2$ with
$$ B(\om)=A + 11\alpha + \frac{1}{2}(11+i\sqrt{11})\beta. $$
Again these conditions are necessary and sufficient. 

\begin{theorem}
The integer group determinants for $\mathbb Z_{11}\rtimes \mathbb Z_5$ are the integers of the form
\be \label{11form}   mN\left(m+11\alpha+\frac{1}{2}(11+i\sqrt{11})\beta \right)^5,\quad 5\nmid m \text{ or } 5^2\mid m,\quad \alpha,\beta\in \mathbb Z. \ee
\end{theorem}


\begin{proof} From Theorem \ref{half} we obtain all \eqref{11form} with $5\nmid m.$
To construct the $25\mid m$ we take
$$ G(x,y)= x^5+ y(x^3-1)-y^2,\quad t(x)=c+a(1-x^5)+b(x^2-x^5),\quad m=0,  $$
in \eqref{shifty}. We have $G(1,y)=1-y^2$ and 
$ A=25c.$
We have 
$$ B_G=\det \begin{pmatrix}  \om^5 & \om^3-1 & -1 & 0 & 0 \\0 & \om^9 & \om-1 & -1 & 0 \\
0 & 0 & \om^3 & \om^4-1 & -1 \\ -1 & 0 & 0 & \om & \om^5-1 \\ \om^9-1 & -1 & 0 & 0 & \om^4 \end{pmatrix} =\frac{1}{2}(11+i\sqrt{11})-11$$
and
$$ \alpha(\om)=\det \begin{pmatrix}  1 & 1 & 1 & 1 & 1\\0 & \om^9 & \om-1 & -1 & 0 \\
0 & 0 & \om^3 & \om^4-1 & -1 \\ -1 & 0 & 0 & \om & \om^5-1 \\ \om^9-1 & -1 & 0 & 0 & \om^4 \end{pmatrix}=\om^9-\om^8+2\om^2+\om +2.$$
Since the squares mod 11 are $1,4,5,9,3$ and the non-squares $2,6,7,8,10$
we get
$$ \sum_{0\leq j\leq 4} \alpha(\om^{4^j})=10 + 2 \cdot   \frac{1}{2} (-1+i\sqrt{11})+  1 \cdot   \frac{1}{2} (-1-i\sqrt{11})=3+\frac{1}{2}(11+i\sqrt{11}).$$
We have 
$$ \lambda_1(\om)=(1-\om^5)\alpha(\om)=\om^9-\om^8-2\om^7-\om^6-2\om^5-\om^3+3\om^2+\om+2,$$
giving
$$ \sum_{0\leq j\leq 4}\lambda_1(\om^{4^j})=10 -1 \cdot   \frac{1}{2} (-1+i\sqrt{11})-1 \cdot   \frac{1}{2} (-1-i\sqrt{11})= 11,$$
and 
$$ \lambda_2(\om)=(\om^2-\om^5)\alpha(\om)=-\om^{10}-2\om^7-\om^6-2\om^5+2\om^4+3\om^2+1,$$
giving
$$ \sum_{0\leq j\leq 4} \lambda_2(\om^{4^j})=5 - 0\cdot   \frac{1}{2} (-1+i\sqrt{11})-1 \cdot   \frac{1}{2} (-1-i\sqrt{11})= \frac{1}{2} (11+i\sqrt{11}).$$
Therefore
$$ B(\om) = 25c+11(a-1-2c)+ \frac{1}{2}(11+i\sqrt{11})(b+c+1) $$
can be made to be of the form $A + 11\alpha +\frac{1}{2}(11+i\sqrt{11})\beta$
for any $\alpha$ and $\beta$ with a suitable choice of $a$ and $b$.
\end{proof}

\subsection{ $\mathbb Z_{13}\rtimes \mathbb Z_6$ or SmallGroup(78,1)}
We have 
$$ \text{SmallGroup}(78,1)=\langle X,Y \; : \; X^{13}=Y^6=1, YX=X^4Y\rangle, $$
and
$$ D=AB^6,    \quad 2\nmid A \text{ or } 2^2\mid A \text{ and } 3\nmid A \text{ or } 3^2\mid A,$$
where $B=N(B(\om))=B(\om)B(\om^2)$ has
$$ B(\om)=A + 13\alpha + \frac{1}{2}(13+\sqrt{13})\beta. $$
Again these conditions are necessary and sufficient. 

\begin{theorem}
The integer group determinants for $\mathbb Z_{13}\rtimes \mathbb Z_6$ are the integers of the form
$$   mN\left(m+13\alpha+\frac{1}{2}(13+\sqrt{13})\beta\right)^6,\quad \text{ $m$ odd or $4\mid m$,} \;\; 3\nmid m \text{ or } 3^2\mid m,\quad \alpha,\beta\in \mathbb Z. $$
\end{theorem}

\begin{proof}
The values with $\gcd(m,6)=1$ follow from  Theorem \ref{half}.

For the multiples of $6$ we take
$$ G(X,Y)=1-Y+(X^{10}-1)Y^3,\quad t(X)=c+(X^3-X^{10})a + (X-X^3)b,\quad m=0, $$
in \eqref{shifty}.
This has $G(1,y)=1-y$ and $A=36c$ and
$$ B_G=\det\begin{pmatrix} 1 & -1 & 0 & \om^{10}-1 & 0 & 0 \\ 0 & 1 & -1 & 0 & \om-1  & 0 \\ 0 & 0 & 1 & -1 & 0 & \om^4-1 \\ \om^3-1 & 0 & 0 & 1 & -1 & 0 \\ 0 & \om^{12}-1 & 0 & 0 & 1 & -1 \\ -1 & 0 & \om^9-1 & 0 & 0 & 1 \end{pmatrix} =-\frac{13}{2}+\frac{\sqrt{13}}{2}, $$ 
and
\begin{align*} \alpha(\om)& =\det \begin{pmatrix} 1 & 1 & 1 & 1 & 1 & 1 \\ 0 & 1 & -1 & 0 & \om-1  & 0 \\ 0 & 0 & 1 & -1 & 0 & \om^4-1 \\ \om^3-1 & 0 & 0 & 1 & -1 & 0 \\ 0 & \om^{12}-1 & 0 & 0 & 1 & -1 \\ -1 & 0 & \om^9-1 & 0 & 0 & 1 \end{pmatrix} \\
& =6 - 4 \om^4 + \om^5 - \om^6 + \om^7 - 2 \om^9 + 2 \om^{10} + 3 \om^{12},
\end{align*}
with
\begin{align*} \lambda_1(\om)& =(\om^3-\om^{10})\alpha(\om)=2 + 4 \om + 2 \om^2 + 7 \om^3 - \om^4 + 2 \om^6 - 6 \om^7 + \om^8 - 4 \om^9 - 5 \om^{10} - 2 \om^{12},   \\
\lambda_2(\om) &   =(\om-\om^3)\alpha(\om)=1 + 6 \om - 3 \om^2 - 6 \om^3 - 4 \om^5 + \om^6 + 3 \om^7 + \om^9 - 3 \om^{10} + 2 \om^{11} + 2 \om^{12}.
\end{align*}
As the squares and non-squares mod 13 are 1,3,4,9,10,12 and 2,5,6,7,8,11 respectively,
\begin{align*}  \sum_{0\leq j\leq 5}\alpha (\om^{4^j}) & =36-\frac{1}{2}(-1+\sqrt{13}) + \frac{1}{2}(-1-\sqrt{13})=36- \sqrt{13},\\
\sum_{0\leq j\leq 5} \lambda_1(\om^{4^j}) & =12-\frac{1}{2}(-1+\sqrt{13}) - \frac{1}{2}(-1-\sqrt{13})=13,\\
\sum_{0\leq j\leq 5} \lambda_2(\om^{4^j}) & =6 - \frac{1}{2}(-1-\sqrt{13})=\frac{1}{2}(13+\sqrt{13}).
\end{align*}
So $B(\om) = 36c +13(a+c-1)+\frac{1}{2}(13+\sqrt{13})(b-2c+1).$ For any $\alpha,\beta$ we get
\be \label{congA} B(\om)=A + 13\alpha +\frac{1}{2}(13+\sqrt{13}) \beta \ee
for suitable choices of $a,b$.

For the multiples of $4$ coprime to 3 we take
$$ G(X,Y)=1+(1-X)Y+Y^2, \quad t(X)=c+ a(X^{11}-1)+b(X^{11}-X^4), \quad m=0,$$
so that $G(1,y)=1+y^2$ and $A=4(1+3c),$ while
$$B_G=\det\begin{pmatrix} 1 & 1-\om & 1 & 0 & 0 & 0  \\ 0 & 1 & 1-\om^4 & 1 & 0 & 0 \\
0 & 0 & 1 & 1-\om^3 & 1 & 0 \\ 0 & 0 & 0 & 1 & 1-\om^{12} & 1 \\ 1 & 0 &  0 & 0 & 1 & 1-\om^9 \\ 1-\om^{10} & 1 & 0 & 0 & 0 & 1 \end{pmatrix}= 4-\frac{1}{2}(13-\sqrt{13}),$$
and
\begin{align*} \alpha(\om) & = \det\begin{pmatrix} 1 & 1 & 1 & 1 & 1 & 1  \\ 0 & 1 & 1-\om^4 & 1 & 0 & 0 \\
0 & 0 & 1 & 1-\om^3 & 1 & 0 \\ 0 & 0 & 0 & 1 & 1-\om^{12} & 1 \\ 1 & 0 &  0 & 0 & 1 & 1-\om^9 \\ 1-\om^{10} & 1 & 0 & 0 & 0 & 1 \end{pmatrix}\\
 &  =\om - \om^3 + \om^6 + \om^7 - \om^9 - \om^{11} + 2 \om^{12},
\end{align*}
with
\begin{align*} \lambda_1(\om)& =(\om^{11}-1)\alpha(\om)=-2 \om + \om^3 + \om^4 + \om^5 - \om^6 - 2 \om^7 + 2 \om^{10} + \om^{11} - \om^{12},   \\
\lambda_2(\om) &   =(\om^{11}-\om^4)\alpha(\om)=1 - \om + \om^2 - 2 \om^3 + \om^4 - \om^9 + \om^{10} - \om^{11} + \om^{12},
\end{align*}
and
\begin{align*}  \sum_{0\leq j\leq 5} \alpha (\om^{4^j}) & =\frac{1}{2}(-1+\sqrt{13}) + \frac{1}{2}(-1-\sqrt{13})=-1,\\
\sum_{0\leq j\leq 5} \lambda_1(\om^{4^j}) & =\frac{1}{2}(-1+\sqrt{13}) - \frac{1}{2}(-1-\sqrt{13})=\sqrt{13},\\
\sum_{0\leq j\leq 5} \lambda_2(\om^{4^j}) & =6 - \frac{1}{2}(-1+\sqrt{13})=\frac{1}{2}(13-\sqrt{13}),
\end{align*}
giving 
$$B(\om)=A+(a-c)\sqrt{13}+\frac{1}{2}(13-\sqrt{13})(b-1-2c).$$
 Suitable $a,b$ 
give \eqref{congA} for any $\alpha,\beta$. The $\pm $ sign lets $A$ take all multiples of $4$ coprime to 3.

For the odd multiples of 9 we take
$$ G(X,Y)=1+XY^2+Y^3,\quad t(X)=c+a(1-X^7)+b(2-X^3-X^7),\quad m=0, $$
giving $G(1,y)=1+y^2+y^3$ and $A=9(1+2c)$, 
$$ B_G=\det\begin{pmatrix} 1 & 0 & \om & 1 & 0 & 0 \\ 0 & 1 & 0 & \om^4 & 1 & 0 \\
0 & 0 & 1 & 0 & \om^3 & 1 \\ 1 & 0 & 0 & 1 & 0 & \om^{12}\\ \om^{9} & 1 & 0 & 0 & 1 & 0 \\
0 & \om^{10} & 1 & 0 & 0 & 1 \end{pmatrix} =9 +\frac{1}{2}(13-\sqrt{13})-13,$$
with
\begin{align*} \alpha(\om) &=\det\begin{pmatrix} 1 & 1 & 1 & 1 & 1 & 1 \\ 0 & 1 & 0 & \om^4 & 1 & 0 \\
0 & 0 & 1 & 0 & \om^3 & 1 \\ 1 & 0 & 0 & 1 & 0 & \om^{12}\\ \om^{9} & 1 & 0 & 0 & 1 & 0 \\
0 & \om^{10} & 1 & 0 & 0 & 1 \end{pmatrix} \\
  & = 1 - \om^2 + \om^5 + \om^6 - \om^9 + \om^{11} + \om^{12}.
\end{align*}
Setting
\begin{align*}
\lambda_1(\om) & =(1-\om^7)\alpha(m)=\om^{11}-\om^7+\om^3-\om^2,\\
\lambda_2(\om) & =(2-\om^3-\om^7)\alpha(\om)=1 - \om - 3 \om^2 + 2 \om^5 + \om^6 - \om^7 - \om^8 - 2 \om^9 + 2 \om^{11} + 2 \om^{12},
\end{align*}
we have
\begin{align*}  \sum_{0\leq j\leq 5} \alpha (\om^{4^j}) & =6+2\cdot \frac{1}{2}(-1-\sqrt{13})=5-\sqrt{13},\\
\sum_{0\leq j\leq 5}\lambda_1(\om^{4^j}) & =\frac{1}{2}(-1+\sqrt{13}) - \frac{1}{2}(-1-\sqrt{13})=\sqrt{13},\\
\sum_{0\leq j\leq 5} \lambda_2(\om^{4^j}) & =6 - \frac{1}{2}(-1+\sqrt{13})=\frac{1}{2}(13-\sqrt{13}),
\end{align*}
giving 
$$B(\om)=A+(a-2c-1)\sqrt{13}+\frac{1}{2}(13-\sqrt{13})(b-1-2c).$$
 Suitable $a,b$ 
give \eqref{congA} for any $\alpha,\beta$, with $A$ any odd multiple of 9. \end{proof}

\section{ The other groups of this form for $p=13$}

\subsection{$\mathbb Z_{13} \rtimes_5 \mathbb Z_4$}
This is the case $p=13$, $n=4$, $t=3$, $r=5$,
$$ G=\langle X, Y \; | \; X^{13}=Y^4=1,\; YX=X^5Y\rangle. $$
We work in  the cubic extension $\mathbb Q(\alpha_i),$
$$ \alpha_1:=\om+\om^5+\om^{12}+\om^8, \quad \alpha_2:=\om^2+\om^{10}+\om^{11}+\om^{3},  \quad \alpha_3:=\om^4+\om^7+\om^{9}+\om^{6}, $$
the roots of $x^3+ x^2-4x+1$, where $\om=e^{2\pi i/13}$.

\begin{theorem} The integer group determinants for $\mathbb Z_{13} \rtimes_5 \mathbb Z_4$
are the
$$ m\: N\left( m+ \sum_{i=1}^3 \beta_i(\alpha_i-4)\right)^4,\quad \text{ $m$ odd or $16\mid m$, } $$
for some $\beta_1,\beta_2,\beta_3$ in $\mathbb Z$, where $N$ is the norm from $\mathbb Q(\alpha_i)$ to $\mathbb Q$.

\end{theorem}

\begin{proof} The odd $m$ were obtained in Lemma \ref{ex}, so we just  have to obtain the values  with $16\mid m$.  We take
$$G=1-xy,\quad t(x)=c + (1-x)\left( \beta_1(x-2x^2+2x^7)+\beta_2(x-x^2+x^7)+\beta_3(x-x^2)\right). $$
We have 
$$B_G=\det \begin{pmatrix} 1 & -\om & 0 & 0 \\ 0 & 1 & -\om^5 & 0 \\ 0 & 0 & 1 & -\om^{12} \\ -\om^8 & 0 & 0 & 1 \end{pmatrix}=0,$$
and
$$\alpha(\om) =\det \begin{pmatrix} 1 & 1 & 1 & 1  \\ 0 & 1 & -\om^5 & 0 \\ 0 & 0 & 1 & -\om^{12} \\ -\om^8 & 0 & 0 & 1 \end{pmatrix}=\om^{12}+\om^8+\om^7+1.$$
Setting
\begin{align*}
\lambda_1(\om)&:= (\om-2\om^2+2\om^7)(1-\om)\alpha(\om) = 2\om^{11}-\om^{10}-2\om^9-\om^8+2\om^6-\om^2+1,\\
\lambda_2(\om)&:= (\om-\om^2+\om^7)(1-\om)\alpha(\om) = \om^{11}-\om^{10}-\om^9+\om^6-\om^2+1,\\
\lambda_3(\om)&:= (\om-\om^2)(1-\om)\alpha(\om) = \om^{11}-\om^{10}-\om^9+\om^8+\om^3-\om^2-\om+1,
\end{align*}
we have
$$ \sum_{j=0}^{3}\alpha(\om^{r^j})=2\alpha_1+\alpha_3 +4=16-2(4-\alpha_1)-(4-\alpha_3),$$
and
$$\sum_{j=0}^{3}\lambda_1(\om^{r^j})=4-\alpha_1, \quad \sum_{j=0}^{3}\lambda_2(\om^{r^j})=4-\alpha_2, \quad \sum_{j=0}^{3}\lambda_1(\om^{r^j})=4-\alpha_3. $$
Hence
$$ A=16c,\;\; B(\om)=16c+(\beta_1-2c)(4-\alpha_1) + \beta_2 (4-\alpha_2)+ (\beta_3-c)(4-\alpha_3) . \qedhere$$
\end{proof}

\subsection{$\mathbb Z_{13}\rtimes \mathbb Z_3$} This is the case $p=13$, $n=3$, $t=4$, $r=3$,
$$ G=\langle X, Y \; | \; X^{13}=Y^3=1,\; YX=X^3Y\rangle. $$
We work in  the quartic extension $\mathbb Q(\alpha_i),$ where
\begin{align*}
\alpha_1 & :=\om + \om^3+\om^9, \quad \alpha_2:=\om^2+\om^6+\om^5, \\
\alpha_3 & := \om^4+\om^{12}+\om^{10}, \quad  \alpha_4:=\om^7+\om^8+\om^{11},\quad \om:=e^{2\pi i/13}. 
\end{align*}
In this case we can explicitly write
\begin{align*}
\alpha_1-3 & = \frac{1}{4}(\sqrt{13}-13) + i \sqrt{ \frac{13-3\sqrt{13}}{8}}, \quad 
 \alpha_2-3 =\frac{1}{4}(-\sqrt{13}-13) + i \sqrt{ \frac{13+3\sqrt{13}}{8}}, \quad\\
\alpha_3-3& = \frac{1}{4}(\sqrt{13}-13) - i \sqrt{ \frac{13-3\sqrt{13}}{8}}, \quad
 \alpha_4-3 =\frac{1}{4}(-\sqrt{13}-13) - i \sqrt{ \frac{13+3\sqrt{13}}{8}}. \quad
\end{align*}
\begin{theorem}
The integer group determinants for $\mathbb Z_{13} \rtimes \mathbb Z_3$
are the
\be \label{thm134} m\: N\left( m+ \sum_{i=1}^4 \beta_i(\alpha_i-3) \right)^3,\quad \text{ $3\nmid m$ or $9\mid m$, } \ee
for some $\beta_1,\ldots,\beta_4$ in $\mathbb Z$, where $N$ is the norm from $\mathbb Q(\alpha_i)$ to $\mathbb Q$.

\end{theorem}
Notice that we could also write \eqref{thm134} as
$$ m\; N\left( m+13\beta_1 +\frac{1}{2}(13+\sqrt{13})\beta_2+ \beta_3 (\alpha_1-3)+\beta_4(\alpha_2-3)\right)^3. $$

\begin{proof} The $m$ with $3\nmid m$ were obtained in Lemma \ref{ex}, so we just  have to obtain the values  with $9\mid m$.  We take
$G=(x^8+x^9-1) - y^2$, $t(x)=c+(x-1) h(x)$ with
\begin{align*} h(x) = & \beta_1(2x^8+x^9+2x^{10}+2x^{11})+\beta_2(x^8+x^9+x^{10}+x^{11})\\ &+\beta_3(x^9+x^{11})+\beta_4(x^5-2x^8-2x^{10}-x^{11}). 
\end{align*}
We have
$$ B_G=\det\begin{pmatrix} \om^8+\om^9-1 & 0 & -1 \\ -1 & \om^{11}+\om-1 & 0 \\ 0 & -1 & \om^7+\om^3-1 \end{pmatrix} = \alpha_1-\alpha_2 $$
and
$$ \alpha(\om) =\det\begin{pmatrix} 1 & 1 & 1 \\ -1 & \om^{11}+\om-1 & 0 \\ 0 & -1 & \om^7+\om^3-1 \end{pmatrix} = -\om^{11} +\om^8+\om^5+\om^4 +1. $$
With
\begin{align*}
\lambda_1(\om) & :=(2\om^8+\om^9+2\om^{10}+2\om^{11})(\om-1)\alpha(\om)\\ & =-3\om^{10}+\om^9-\om^8+\om^7+2\om^6-\om^5+3\om^4-\om^2-1,\\
\lambda_2(\om) & :=(\om^8+\om^9+\om^{10}+\om^{11})(\om-1)\alpha(\om) =-\om^{10}-\om^8+\om^7+\om^6+\om^4-1,\\
\lambda_3(\om) & :=(\om^9+\om^{11})(\om-1)\alpha(\om) =\om^{12}-\om^{11}-\om^8+2\om^7-\om^6+\om^5-1,\\
\lambda_4(\om) & :=(\om^5-2\om^8-2\om^{10}-\om^{11})(\om-1)\alpha(\om)\\ & =\om^{12}+3\om^{10}-2\om^9+\om^7-2\om^6+\om^5-4\om^4+\om^3+\om^2+\om -1,\\
\end{align*}
we have
$$ \sum_{j=0}^2 \alpha(\om^{r^j})=\alpha_2+\alpha_3+3=9+(\alpha_2-3)+(\alpha_3-3), \quad  \sum_{j=0}^2 \lambda_i(\om^{r^j})=(\alpha_i-3), \; i=1,\ldots ,4. $$
Hence $A=9c$ and
$$ B(\om)=9c+ \beta_1(\alpha_1-3)+    (\beta_2+c)(\alpha_2-3) + (\beta_3+c)(\alpha_3-3) +\beta_4(\alpha_4-3).  \qedhere$$
\end{proof}

\section{Speculations}

From our admittedly limited  number of small examples, it is tempting to ask:

\begin{question}

For $n=p-1$, $G=GA(1,p),$ do the integer group determinants achieve all integers of the form
\be \label{Q1} m(m+\ell p)^{p-1},\quad l\in \mathbb Z,  \ee
with $m$ a $\mathbb Z_n$ integer determinant?
\end{question}
\vskip0.1in

\begin{question}
 For $n=\frac{1}{2}(p-1)$, $G=\mathbb Z_p \rtimes \mathbb Z_{\frac{1}{2}(p-1)},$ do the integer group determinants achieve all integers of the form
\be \label{Q2}  m \; N\left(m+\alpha p + \frac{1}{2}(p+\sqrt{\ve p})\beta\right)^{\frac{1}{2}(p-1)},\quad \alpha,\beta \in \mathbb Z,  \ee
with $m$ a $\mathbb Z_n$ integer determinant?
\end{question}

If not in general true, are these at least true when $p=2q+1$ with $q$ a Sophie Germain prime? In that case we just need to obtain \eqref{Q1} for the $m$ with $4\mid m$, $q\nmid m$ or $q^2\mid m$, $2\nmid m$ and \eqref{Q2} for the $m$ with $q^2\mid m$.
 
For general $t$ we can also ask whether the form of the integer group determinant  given in Theorem \ref{generalt} is always if and only if
(though it is not clear which integers can be achieved with norms of the given form).

\section*{Acknowledgement}
This research began as a project with Humberto Bautista Serrano

\end{document}